\documentclass[12pt, reqno]{amsart}
\usepackage{amsmath, amsthm, amscd, amsfonts, amssymb, graphicx}
\usepackage[bookmarksnumbered, colorlinks, plainpages]{hyperref}
\input{mathrsfs.sty}

\newtheorem{theorem}{Theorem}[section]
\newtheorem{lemma}[theorem]{Lemma}
\newtheorem{proposition}[theorem]{Proposition}
\newtheorem{corollary}[theorem]{Corollary}
\theoremstyle{definition}
\newtheorem{definition}[theorem]{Definition}
\newtheorem{example}[theorem]{Example}

\theoremstyle{remark}

\numberwithin{equation}{section}

\begin{document}

\title[Moore-Penrose inverse of Gram operator]
{Moore-Penrose inverse of Gram operator on Hilbert $C^*$-modules}
\author{M. S. Moslehian}
\address{Mohammad Sal Moslehian: Department of Pure Mathematics,
Center of Excellence in Analysis on Algebraic Structures (CEAAS),
Ferdowsi University of Mashhad, P.O. Box 1159, Mashhad 91775, Iran.}
\email{moslehian@ferdowsi.um.ac.ir and moslehian@member.ams.org}
\urladdr{http://profsite.um.ac.ir/~moslehian/}
\author{K. Sharifi}
\address{Kamran Sharifi: Department of Mathematics,
Shahrood University of Technology, P. O. Box 3619995161-316,
Shahrood, Iran.
\newline School of Mathematics, Institute for Research in
 Fundamental Sciences (IPM), P. O. Box: 19395-5746, Tehran, Iran.}
\email{sharifi.kamran@gmail.com}
\author{M. Forough}
\address{Marzieh Forough: Department of Mathematics,
Ferdowsi University of Mashhad, P. O. Box 1159, Mashhad, Iran.}
\email{mforough86@gmail.com}
\author{M. Chakoshi}
\address {Mahnaz Chakoshi: Tusi Mathematical Research Group,
Mashhad, P.O. Box 1113, Iran.}
 \email{ma\_chakoshi@yahoo.com}


\begin{abstract}
Let $t$ be a regular operator between Hilbert $C^*$-modules and
$t^\dag$ be its Moore-Penrose inverse. We investigate the
Moore-Penrose invertibility of the Gram operator $t^*t$. More
precisely, we study some conditions ensuring that $t^{ \dag} = (t^*
\, t)^{ \dag} \, t^*= t^*\,(t \, t^*)^{ \dag}$ and $(t^*t)^{
\dag}=t^{ \dag}t^{* \, \dag}$ hold. As an application, we get some
results for densely defined closed operators on Hilbert
$C^*$-modules over $C^*$-algebras of compact operators.
\end{abstract}
\maketitle

\section{Introduction.}
Hilbert $C^*$-modules are essentially objects like Hilbert spaces,
except that the inner product, instead of being complex-valued,
takes its values in a $C^*$-algebra. Although Hilbert $C^*$-modules
behave like Hilbert spaces in some ways, some fundamental Hilbert
space properties like Pythagoras' equality, self-duality, and even
decomposition into orthogonal complements do not hold in general. A
(right) {\it pre-Hilbert $C^*$-module} over a $C^*$-algebra
$\mathscr{A}$ is a right $\mathscr{A}$-module $\mathscr{X}$ equipped
with an $\mathscr{A}$-valued inner product $\langle \cdot , \cdot
\rangle : \mathscr{X} \times \mathscr{X} \to \mathscr{A}\,, \ (x,y)
\mapsto \langle x,y \rangle$, which is $\mathscr A$-linear in the
second variable $y$ as well as $ \langle x,y \rangle=\langle y,x
\rangle ^{*}$ and $\langle x,x \rangle \geq 0$ with equality only
when $x=0$. A pre-Hilbert $\mathscr{A}$-module $\mathscr{X}$ is
called a \emph{Hilbert $ \mathscr{A}$-module} if $\mathscr{X}$ is a
Banach space with respect to the norm $\| x \|=\|\langle x,x\rangle
\| ^{1/2}$, where the latter norm denotes the norm of ${\mathscr
A}$. Each $C^*$-algebra ${\mathscr A}$ can be regarded as a Hilbert
${\mathscr A}$-module via $\langle a, b\rangle = a^*b\,\,(a, b \in
{\mathscr A})$. A Hilbert $\mathscr A$-submodule $W$ of a Hilbert
$\mathscr A$-module $\mathscr{X}$ is an orthogonal summand if $W
\oplus W^\bot = \mathscr{X}$, where $W^\bot$ denotes the orthogonal
complement of $W$ in $\mathscr{X}$.

Throughout this paper we assume that $\mathscr{A}$ is an arbitrary
$C^*$-algebra (not necessarily unital) and $\mathscr{X},
\mathscr{Y}$ are Hilbert $\mathscr A$-modules. By an operator we
mean a linear operator. We may deal with bounded and unbounded
operators at the same time, so we will denote bounded operators by
capital letters and unbounded operators by small letters. In
addition, ${\rm Dom}(\cdot)$, ${\rm Ker}(\cdot)$ and ${\rm
Ran}(\cdot)$ stand for domain, kernel and range of operators,
respectively. An operator $t$ between $\mathscr{X}$ and
$\mathscr{Y}$ is a linear operator with ${\rm Dom}(t) \subseteq
\mathscr{X}$ and ${\rm Ran}(t) \subseteq \mathscr{Y}$. It is
$\mathscr{A}$-linear if $t(xa)=t(x)a$ for all $x \in {\rm Dom}(t)$
and all $a \in \mathscr{A}$. The set of all $\mathscr{A}$-linear
operators between $\mathscr{X}$ and $\mathscr{Y}$ is denoted by
$\mathcal{L}(\mathscr{X},\mathscr{Y})$. As usual,
$\mathcal{L}(\mathscr{X})$ stands for
$\mathcal{L}(\mathscr{X},\mathscr{Y})$ if $\mathscr{X}=\mathscr{Y}$.
An operator $t\in \mathcal{L}(\mathscr{X},\mathscr{Y})$, whose ${\rm
Dom}(t)$ is a dense submodule of $\mathscr{X}$ is called a
\emph{densely defined operator}. An operator $t\in
\mathcal{L}(\mathscr{X},\mathscr{Y})$ is called \emph{closed} if its
graph $G(t)=\{(x,t(x)):~x \in {\rm Dom}(t)\}$ is a closed submodule
of the Hilbert $\mathscr{A}$-module $\mathscr{X} \oplus \mathscr{Y}$
equipped with the $C^*$-inner product $\langle (x_1, y_1), (x_2,
y_2)\rangle = \langle x_1, x_2\rangle+\langle y_1, y_2\rangle$. If
$s\in \mathcal{L}(\mathscr{X},\mathscr{Y})$ is an extension of $t\in
\mathcal{L}(\mathscr{X},\mathscr{Y})$, we write $t\subset s$. As
usual, this means that ${\rm Dom}(t) \subseteq {\rm Dom}(s)$ and
$s(x)=t(x)$ for all $x\in {\rm Dom}(t)$. If $t$ has a closed
extension, then it is called \emph{closable} and the operator
$\overline{t}\in \mathcal{L}(\mathscr{X},\mathscr{Y})$ with the
property $G(\overline{t}) = \overline{G(t)}$ is called the
\emph{closure} of $t$. A densely defined operator $t\in
\mathcal{L}(\mathscr{X},\mathscr{Y})$ is called \emph{adjointable}
if there exists a densely defined operator $t^*\in
\mathcal{L}(\mathscr{Y},\mathscr{X})$ with the domain $ {\rm
Dom}(t^*)=\{y \in \mathscr{Y}:~ \mbox{there exist}~ z \in
\mathscr{X} ~ \mbox{such that}~ \langle t(x),y\rangle =\langle
x,z\rangle ~\mbox{for any~} x \in {\rm Dom}(t)\} $ satisfying the
property $\langle t(x),y\rangle=\langle x,t^*(y)\rangle$ for any $x
\in {\rm Dom}(t),~ y\in {\rm Dom}(t^*)$. This property ensures that
${\rm Dom}(t^*)$ is a submodule of $\mathscr{Y}$ and $t^*$ is a
closed $\mathscr{A}$-linear map. In the setting of Hilbert spaces
any densely defined closed operator has a densely defined adjoint
but this phenomenon does not occur in the framework of Hilbert
$C^*$-modules in general. It is notable that any adjointable
operator with domain $\mathscr{X}$ is automatically a bounded
$\mathscr{A}$-linear map. We denote by
$\mathbb{B}(\mathscr{X},\mathscr{Y})$ the set of all adjointable
operators from $\mathscr{X}$ into $\mathscr{Y}$. The set
$\mathbb{B}(\mathscr{X},\mathscr{X})$ is abbreviated by
$\mathbb{B}(\mathscr{X})$.

If $s \in \mathcal{L}(\mathscr{X},\mathscr{Y})$ and $t\in
\mathcal{L}(\mathscr{Y},\mathscr{Z})$ are densely defined operators
between Hilbert $C^*$-modules, we define the composition operator
$ts$ by $(ts)(x) = t(s(x))$ for all $x \in {\rm Dom}(ts)$, where
${\rm Dom}(ts)= \{x \in {\rm Dom}(s): s(x) \in {\rm Dom}(t)\}$. Then
$ts \in \mathcal{L}(\mathscr{X},\mathscr{Z})$, but $ts$ is not
necessarily densely defined. Suppose two densely defined operators
$t$ and $s$ are adjointable, then it is easy to see that $s^*t^*
\subset (ts)^*$. If $T$ is a bounded adjointable operator, then
$s^*T^* = (Ts)^*$. Damaville \cite{DAM} proved that under certain
conditions the product of two regular operators between Hilbert
$C^*$-modules is regular. Regular operators on Hilbert $C^*$-modules
were first introduced by Baaj and Julg \cite{BAA} and extensively
studied in \cite{LAN}.

\begin{definition}
An operator $t\in \mathcal{L}(\mathscr{X},\mathscr{Y})$ is said to
be \emph{regular} if  $t$ is densely defined, closed and adjointable
and the range of $1+t^*t$ is dense in $\mathscr{X}$. We denote the
set of all regular operators in
$\mathcal{L}(\mathscr{X},\mathscr{Y})$ by
$\mathcal{R}(\mathscr{X},\mathscr{Y})$. We abbreviate
$\mathcal{R}(\mathscr{X},\mathscr{X})$ by
$\mathcal{R}(\mathscr{X})$.
\end{definition}
This definition is equivalent to the notion of regularity introduced
by Woronowicz \cite{wor}. If $t$ is regular, then $t^*$ is regular,
$t=t^{\ast\ast}$ and also $t^*t$ is regular and self-adjoint. It may
occur that $t^*$ is regular but not $t$, see \cite[Propositions 2.2
and 2.3]{PAL}. Also a densely defined operator $t$ is regular if and
only if its graph is orthogonally complemented in $\mathscr{X}
\oplus \mathscr{Y}$, cf. \cite[Corollary 3.2]{F-S}. Suppose $t \in
\mathcal{R}(\mathscr{X},\mathscr{Y})$ and define
$Q_{t}=(1+t^*t)^{-1/2}$ and $F_{t}=tQ_{t}$. Then ${\rm
Ran}(Q_{t})={\rm Dom}(t)$,  $0 \leq Q_{t}= (1 - F_{t}^{*}
F_{t})^{1/2} \leq 1$ in $\mathbb{B}(\mathscr{X})$ and $F_{t}\in
\mathbb{B}(\mathscr{X},\mathscr{Y})$.

The following terminology is basic in our study, cf. \cite{FS2}.

\begin{definition}\label{3}
Let $t \in \mathcal{R}(\mathscr{X},\mathscr{Y})$. An operator $s \in
\mathcal{R}(\mathscr{Y},\mathscr{X})$ is called a
\emph{Moore--Penrose inverse} of $t$ if $tst=t,~sts=s,~(ts)^*=
\overline{ts}$ and $(st)^*=\overline{st}$.
\end{definition}

If $t \in \mathcal{L}(\mathscr{H},\mathscr{H}')$ is a densely
defined closed operator between Hilbert spaces, then Py'tev
\cite{PYT} proved that there is a densely defined closed operator
$s$ satisfying the relations in Definition \ref{3}. Xu and Sheng
\cite{xu} proved that an adjointable operator acting on the whole of
a Hilbert $C^*$-module has a Moore--Penrose inverse if and only if
it has closed range. In \cite[Theorem 3.1]{FS2}, a very useful
necessary and sufficient condition for a regular operator $t$ to
admit a unique Moore--Penrose inverse, denoted by $t^\dagger$, is
given as follows:

\begin{theorem} \label{thm_polar_decomp}
If $t\in \mathcal{R}(\mathscr{X},\mathscr{Y})$, then the following
conditions are equivalent:
\newcounter{cou001}
\begin{list}{(\roman{cou001})}{\usecounter{cou001}}

\item $t$ and $t^*$ have unique Moore-Penrose inverses
which are adjoint to each other, $ t ^{ \dag}$ and $t ^{ \dag \,
* }$.

\item $\mathscr{X}={\rm Ker}(t) \oplus \overline{{\rm Ran}(t^*)}$ and $\mathscr{Y}={\rm Ker}(t^*) \oplus
\overline{{\rm Ran}(t)}$.
\end{list}

\noindent In this situation, $ \overline{t^* \, t ^{ \dag \, *}}$
and $\overline{t t ^{ \dag }}$ are the projections onto $
\overline{{\rm Ran}(t^*)}=\overline{{\rm Ran}(t^*t)}$ and $
\overline{{\rm Ran}(t)}$, respectively.
\end{theorem}
Recall that ${\rm Dom}(t^{ \dag}):={\rm Ran}(t) \oplus {\rm
Ker}(t^*)$ and $t^{ \dag}: {\rm Dom}(t^{ \dag})\subseteq \mathscr{Y}
\to \mathscr{X}$ is defined by $t^{ \dag}
\,(t(x_{1}+x_{2})+x_{3})=x_{1}$, for all $x_{1} \in Dom(t) \cap
\overline{{\rm Ran}(t^*)}$, $x_{2} \in {\rm Dom}(t) \cap {\rm
Ker}(t)$ and $x_{3} \in {\rm Ker}(t^*)$. The adjoint of $t^{ \dag}$
is defined similarly.

In view of \cite[Corollary 3.4 ]{FS2}, every regular operator with
closed range has a bounded adjointable Moore-Penrose inverse.

Let $T$ be a bounded linear operator with closed range between
Hilbert spaces. The \emph{Gram operator} of $T$ is defined to be the
operator $T^*T$. One of interesting problems in matrix
theory/operator theory is investigation of the Moore--Penrose
inverse $(T^*T)^\dagger$. The equalities $(T^*T)^\dagger=T^\dagger
T^*{^\dagger}$ and $T^\dagger=T^*(TT^*)^\dagger=(T^*T)^\dagger T^*$
were presented in \cite{DDM, Groetsch1995} in the case when $T$ is a
bounded linear operator acting on a Hilbert space. In this paper, we
generalize them to the case where $t$ is a certain operator in the
framework of Hilbert $C^*$-modules. We set conditions which ensure
that $t^*\,(t \, t^*)^{ \dag} = t^{ \dag}$
 and $ t^{ \dag} = (t^* \, t)^{ \dag} \, t^*$ and
 $(t^*t)^{ \dag}=t^{ \dag}t^{* \, \dag}$ hold.
We present an example showing that the equalities do not hold in
general. Finally, we apply our results to densely defined closed
operators on Hilbert $C^*$-modules over $C^*$-algebras of compact
operators.

\section{Moore-Penrose invertibility of Gram operator.}

In this section we obtain unbounded versions of some results of
\cite{DDM} in the framework of Hilbert $C^*$-modules. Indeed, we
study some conditions which ensure that $t^{ \dag} = (t^* \, t)^{
\dag} \, t^*= t^*\,(t \, t^*)^{ \dag}$ and $(t^*t)^{ \dag}=t^{
\dag}t^{* \, \dag}$ hold. Our results are also reformulated in terms
of bounded adjointable operators.

\begin{lemma} \label{Forough0}
If $t \in \mathcal{R}(\mathscr{X},\mathscr{Y})$ has closed range,
then so does $t \, t^*$.
\end{lemma}
\begin{proof} If $ {\rm Ran}(t)$ is closed, then $ {\rm Ran}(t^*)$ is closed and $\mathscr{X}=
{\rm {\rm Ker}}(t) \oplus {\rm Ran}(t^*)$. Let $x \in {\rm Dom}(t)$.
Then $x=z + t^*y $, for some $y \in {\rm Dom}(t^*)$ and $z \in {\rm
Ker}(t) \subseteq {\rm Dom}(t)$. We therefore have $t (x)= t\, t^*
(y)$, that is, $t$ and $t \, t^*$ have the same range.
\end{proof}

\begin{theorem} \label{Forough1}
Suppose $t \in \mathcal{R}(\mathscr{X},\mathscr{Y})$ and
$\overline{{\rm Ran}(t^*)}$ and $ \overline{{\rm Ran}(t)}$ are
orthogonally complemented in $\mathscr{X}$ and $\mathscr{Y}$,
respectively. Then the following statements hold:
\begin{list}{(\roman{cou001})}{\usecounter{cou001}}
\item $t^*\,(t \, t^*)^{ \dag} \subseteq t^{ \dag}$,
\item $ t^{ \dag} \subseteq (t^* \, t)^{ \dag} \, t^*$ if and only
if ${\rm Ran}(t) \subseteq {\rm Dom}(t^*)$.
\end{list}
 If in addition $t$ has closed range, then
\begin{list}{(\roman{cou001})}{\usecounter{cou001}}
   \setcounter{cou001}{2}
\item $t^*\,(t \, t^*)^{ \dag} = t^{ \dag}$,
\item $ t^{ \dag} = (t^* \, t)^{ \dag} \, t^*$ when
${\rm Ran}(t) \subseteq {\rm Dom}(t^*)$.
\end{list}
\end{theorem}

\begin{proof} The Moore-Penrose inverse of the regular operator $t \, t^*$ exists by
the orthogonal decompositions into direct summands and the fact that
$ \overline{{\rm Ran}( t)}= \overline{{\rm Ran}( t \,t^*)}$. To
prove (i) we have $ {\rm Dom} ( t^* (t \, t^*)^{ \dag} ) = {\rm Dom}
( (t \, t^*)^{ \dag} )= {\rm Ran}(t \, t^*)+{\rm Ker}(t \, t^*)
\subseteq { \rm Ran}(t)+{ \rm Ker}(t^*)= { \rm Dom}( t^{ \dag} ).$
Let $x=t \, t^*(x_1 +x_2)+x_3 \in {\rm Dom} ( (t \, t^*)^{ \dag} )$
with $x_1 \in {\rm Dom} ( t \, t^* ) \cap \overline{ {\rm Ran}(  t
\, t^*) }$, $x_2 \in {\rm Dom} ( t \, t^* ) \cap  {\rm Ker}(  t \,
t^*)$ and $x_3 \in {\rm Ker}(  t \, t^*)= {\rm Ker}( t^*)= {\rm
Ker}( t^{ \dag})$. Then $(t \, t^*)^{\dag} ( x) =x_1$. We therefore
have $$t^{ \dag}(x)= t^{ \dag}(t \, t^*(x_1 +x_2)+x_3)= (t^{ \dag} t
)\, t^* (x_1)+0+0= t^*(x_1)= t^*\,(t \, t^*)^{ \dag}(x).$$ That is,
$t^*\,(t \, t^*)^{ \dag}= t^{ \dag}$ on  $ {\rm Dom} ( t^* (t \,
t^*)^{ \dag} )$.

Let the operator inclusion of (ii) hold. Then ${\rm Dom}(t^{ \dag})
\subseteq {\rm Dom}(t^{*})$, which implies that ${\rm Ran}(t)
\subseteq {\rm Dom}(t^*)$. Conversely, if ${\rm Ran}(t) \subseteq
{\rm Dom}(t^*)$ and $x= t (x_1 +x_2)+x_3 \in {\rm Dom}( t^{ \dag})$
 with  $x_1 \in {\rm Dom} ( t) \cap \overline{ {\rm Ran}( t^*) }$, $x_2 \in
{\rm Dom} ( t ) \cap  {\rm Ker}(  t )$ and $x_3 \in {\rm Ker}(
t^*)$, then $t^{ \dag} (x)=x_1$. Since $ \overline{{\rm Ran}( t^*)}=
\overline{{\rm Ran}( t^* \,t)}$, we get
$$ ((t^* \, t)^{ \dag}t^*)(t (x_1 +x_2)+x_3)=
(t^* \, t)^{ \dag} \, (t^* \,t) \,x_1+0+0= \overline{(t^* \, t)^{
\dag} \, (t^* \,t)} \,x_1 =  x_1=t^{ \dag}(x) .$$ That is,
$(t^*\,t)^{ \dag} \, t^*= t^{ \dag}$ on  $ {\rm Dom} ( t^{ \dag} )$.

To demonstrate (iii) we suppose that $t$ has closed range. Then $t
\, t^*$ has closed range and ${\rm Ran}(t \, t^*)= {\rm Ran}(t)$. In
this case, $t^{ \dag}$ is everywhere defined. Hence, $$ {\rm Dom} (
t^* (t \, t^*)^{ \dag} ) = {\rm Ran}(t \, t^*)+{\rm Ker}(t \, t^*) =
{ \rm Ran}(t)+{ \rm Ker}(t^*)= { \rm Dom}( t^{ \dag}
)=\mathscr{Y}.$$ The result now follows from (i). Finally, if $t^{
\dag}$ is everywhere defined and ${\rm Ran}(t) \subseteq {\rm
Dom}(t^*)$, then the inclusion of (ii) changes to an equality which
completes the proof.
\end{proof}

\begin{corollary} \label{Forough001}
Suppose $t \in \mathcal{R}(\mathscr{X},\mathscr{Y})$ has closed
range and ${\rm Ran}(t) \subseteq {\rm Dom}(t^*)$. Then
\begin{list}{(\roman{cou001})}{\usecounter{cou001}}
\item $t^{ \dag} = (t^* \, t)^{ \dag} \, t^*= t^*\,(t \, t^*)^{
\dag}$,
\item $(t^*t)^{ \dag}=t^{ \dag}t^{* \, \dag}$. 
\end{list}
\end{corollary}
\begin{proof} According
to \cite[Proposition 1.2]{F-S} and \cite[Theorem 3.2]{LAN}, ${\rm
Ran}(t)$ and ${\rm Ran}(t^*)$ are orthogonally complemented in
$\mathscr{Y}$ and $\mathscr{X}$, respectively. These facts together
with Theorem \ref{Forough1} imply the equalities of the first part.

The closedness of the range of $t^*$ and part (iii) of Theorem
\ref{Forough1} ensure that $t^{* \, \dag} = t \, (t^* \, t)^{
\dag}$. We therefore have
$$ t^{ \dag} t^{* \, \dag}= (t^*t)^{ \dag} t^* \, t (t^*t)^{ \dag}=
(t^*t)^{ \dag}.$$
\end{proof}

\begin{corollary} \label{Forough0001}
Suppose $T \in \mathbb{B}(\mathscr{X},\mathscr{Y})$ has closed
range. Then
\begin{list}{(\roman{cou001})}{\usecounter{cou001}}
\item $T^{ \dag} = (T^* \, T)^{ \dag} \, T^*= T^*\,(T \, T^*)^{
\dag}$,
\item $(T^*T)^{ \dag}=T^{\, \dag}T^{* \, \dag}$.
\end{list}
\end{corollary}

The preceding result follows from the fact that
$\mathbb{B}(\mathscr{X},\mathscr{Y})$ is a subset of
$\mathcal{R}(\mathscr{X},\mathscr{Y})$. The set of all regular
operators from a topological point of view were studied in
\cite{SHA/GAP, SHA/APROACH, SHA/POLAR}.

The reader should be aware that the operator inclusions given in
Theorem \ref{Forough1} may be strict even for bounded operators.

\begin{example} Let $V$ be the Voltera operator on
$L^2[0,1]$, i.e., $(Vf)(x)= \int_0^x \, f(y) \, \mathrm{d}y$. Then
the adjoint of $V$ is given by $(V^*f)(x)= \int_x^1 \, f(y) \,
\mathrm{d}y$. The operators $V$ and $V^*$ are bounded and
injective, that is, ${ \rm Ker}(V)={ \rm Ker}(V^*)=\{ 0 \}$ and
${\rm Ran}(V)$ and ${\rm Ran}(V^*)$ are dense in $L^2[0,1]$. Indeed,
$$L^2[0,1]={ \rm Ker}(V^*) \oplus \overline{ {\rm Ran}(V)}= \overline{{\rm Ran}(V)},$$
$$L^2[0,1]={ \rm Ker}(V) \oplus \overline{ {\rm Ran}(V^*)}= \overline{{\rm Ran}(V^*)}.$$
Moreover, we have $(V \, V^*f)(x)= \int_0^x \, (\int_y^1 f(t) \,
\mathrm{d}t) \,   \mathrm{d}y = x \int_x^1 f(y)\, \mathrm{d}y  +
\int_0^x \, yf(y) \, \mathrm{d}y.$ We claim that $ {\rm Ran}(V \,
V^*) \neq {\rm Ran}(V)$. To see this, we consider the identity
function $f(x)=x$ in ${\rm Ran}(V)$. If $f=V\, V^*g$ for some $g \in
L^2[0,1]$, then
$$f'(x)=\frac{ \mathrm{d}}{ \mathrm{d}x} \ (\int_0^x \, (\int_y^1 g(t) \,
\mathrm{d}t) \, \mathrm{d}y)= \int_x^1 g(t) \, \mathrm{d}t,~~~
\mathrm{for}~ \mathrm{ all}~ x \in[0,1],$$ which implies that
$f'(1)=0$, a contradiction. This means that the
Volterra integral equation $x = x \int_x^1 g(y)\, \mathrm{d}y  +
\int_0^x \, yg(y) \, \mathrm{d}y$ has no solution in
$L^2[0,1]$.

Since ${\rm Dom}(V^*(V \,V^*)^{ \dag})= {\rm Dom}((V \,V^*)^{
\dag})= {\rm Ran}(V \, V^*)+{\rm Ker}(V \,V^*)={\rm Ran}(V \,
V^*)+{\rm Ker}(V^*)= {\rm Ran}(V \, V^*)$, we have ${\rm Dom}(V^*(V
\,V^*)^{ \dag} \subseteq {\rm Ran}(V)= {\rm Ran}(V)+{\rm Ker}(V^*) =
{\rm Dom}(V^{ \dag})$. The latter inclusion is strict since $ {\rm
Ran}(V \, V^*) \subset {\rm Ran}(V)$. Hence, ${\rm Dom}(V^*(V
\,V^*)^{ \dag}) \subset {\rm Dom}(V^{ \dag})$. This means that
the operator inclusions given in Theorem \ref{Forough1} may be
strict even for bounded operators on Hilbert spaces. Another example
can be found in the book of Ben-Israel and Greville, cf.
\cite[Chapter 9, Ex. 26]{Ben-Israel}.
\end{example}

The above example also shows that the assumption ``closedness of the
range of $t$'' in  part (ii) of Corollary \ref{Forough001} cannot be
removed.

\begin{theorem} \label{Forough4}
Suppose $t \in \mathcal{R}(\mathscr{X})$ and $\overline{{\rm
Ran}(t^*)}$ and $ \overline{{\rm Ran}(t)}$ are orthogonally
complemented in $\mathscr{X}$. If $S$ is a bounded adjointable
operator which commutes with $t$ and $t^*$, then $S t^{ \dag}
\subseteq t^{ \dag}S$.
\end{theorem}
\begin{proof} Suppose $S$ commutes with $t$ and $t^*$ and $ \omega > 0$.
Then $S$ commutes with $( \omega 1+ tt^*)$, which implies
commutativity of $S$ and of the bounded operator $( \omega 1+
tt^*)^{-1}$. In view of commutativity of $S$ with $t^*$ and $(
\omega 1+ tt^*)^{-1}$, boundedness of $S$ and Theorem 2.8 of
\cite{SHA/Groetsch}, we infer that
$$S \, t^{ \dag}=S \, \lim_{ \omega \to 0^+} t^* ( \omega 1+
tt^*)^{-1}= \lim_{ \omega \to 0^+} t^* ( \omega 1+ tt^*)^{-1} \, S=
t^{ \dag} \, S ~~{\rm~on~ Dom}( t^{ \dag}).$$
\end{proof}

\begin{proposition} \label{Forough6} Suppose $t \in \mathcal{R}(\mathscr{X})$
and $\overline{{\rm Ran}(t^*)}$ and $ \overline{{\rm Ran}(t)}$ are
orthogonally complemented in $\mathscr{X}$. Then $t$ is selfadjoint
if and only if $t= \overline{ t^{ \dag}t} \, t^*$.
\end{proposition}
\begin{proof} The assertion follows from the fact that
$ \overline{ t^{ \dag}t}$ is the orthogonal projection onto $
\overline{ {\rm Ran}(t^{ \dag})}= {\rm Ker}(t^{ \dag \, *})^{
\perp}={\rm Ker}(t)^{ \perp}= \overline{ {\rm Ran}(t^*)}$, which
implies $t^*= \overline{ t^{ \dag}t} \, t^*$.
\end{proof}

We end our paper with the following useful observations.
By an arbitrary $C^*$-algebra of compact
operators $\mathscr{A}$ we mean that
$\mathscr{A} = c_{0}$-$ \oplus_{i \in I}\mathbb{K}
(\mathscr{H}_{i})$, i.e., $\mathscr{A}$ is a
$c_{0}$-direct sum of elementary $C^*$-algebras $\mathbb{K}(\mathscr{H}_{i})$ of all
compact operators acting on Hilbert spaces $\mathscr{H}_{i}, \ i \in
I$. If $\mathscr{A}$ is an arbitrary $C^*$-algebra of compact
operators, then for every pair of Hilbert $ \mathscr{A}$-modules
$\mathscr{X}, \mathscr{Y}$, every densely defined closed operator
$t: Dom(t)\subseteq \mathscr{X} \to \mathscr{Y}$ is automatically
regular and has a Moore-Penrose inverse, cf. \cite{F-S, FS2, FR1, GUL}.
The following results follow from \cite[Theorem 3.8]{FS2}.

\begin{corollary} \label{Forough11}
Suppose $\mathscr{X}$ and $\mathscr{Y}$ are Hilbert $C^*$-modules
over an arbitrary $C^*$-algebra of compact operators and $t\in
\mathcal{L}(\mathscr{X},\mathscr{Y})$ is a densely defined closed
operator. Then the conclusions of Theorems \ref{Forough1}, \ref{Forough4}
and Proposition \ref{Forough6} hold.
\end{corollary}

{\bf Acknowledgement.} The first author was supported by a grant
from Ferdowsi University of Mashhad (No. MP91267MOS). The work of
the second author was in part supported by a grant from IPM (No.
90470018). The work of the third author was done during her stay at
the International School for Advanced Studies (SISSA), Trieste,
Italy, in 2012. She would like to express her thanks for their warm
hospitality. The authors would like to sincerely thank the anonymous
referee for useful suggestions improving the paper.

{ \bf Addendum.} The first and the fourth authors and the second and
the third authors have investigated the same problem separately.
Based on a suggestion made by the referee of both papers, all authors
decided to combine their manuscripts in the current joint work.

\end{document}